\documentclass[11pt]{article}

\usepackage{amssymb}
\usepackage[intlimits]{amsmath}
\usepackage{amsfonts}
\usepackage{amsthm}
\usepackage{cite}
\usepackage{mathptmx}

\usepackage{hyperref}

\usepackage[a4paper]{geometry}

\let\oldsqrt\sqrt
\def\sqrt{\mathpalette\DHLhksqrt}
\def\DHLhksqrt#1#2{%
\setbox0=\hbox{$#1\oldsqrt{#2\,}$}\dimen0=\ht0
\advance\dimen0-0.2\ht0
\setbox2=\hbox{\vrule height\ht0 depth -\dimen0}%
{\box0\lower0.4pt\box2}}

\allowdisplaybreaks

\newcommand{\R}{\mathbb{R}} 
\newcommand{\N}{\mathbb{N}} 
\newcommand{\dist}{\textnormal{dist}} 
\newcommand{\diam}{\textnormal{diam}} 
\newcommand{\supp}{\textnormal{supp}} 
\DeclareMathOperator{\essinf}{essinf} 

\renewcommand{\phi}{\varphi}

\newcommand{\cC}{{\mathcal C}}
\newcommand{\cD}{{\mathcal D}}

\newcommand{\cH}{{\mathcal H}}

\newcommand{\cJ}{{\mathcal J}}

\newcommand{\eps}{\varepsilon}

\theoremstyle{definition}
\newtheorem{defi}{Definition}[section]
\newtheorem{bem}[defi]{Remark}

\theoremstyle{plain} 
\newtheorem{satz}[defi]{Theorem}

\newtheorem{prop}[defi]{Proposition}
\newtheorem{lemma}[defi]{Lemma}
\newtheorem{cor}[defi]{Corollary}

\theoremstyle{definition}

\numberwithin{equation}{section}
\title{Symmetry via antisymmetric maximum principles in nonlocal problems of variable order}
\author{Sven Jarohs\footnote{Goethe-Universit\"at, Frankfurt, jarohs@math.uni-frankfurt.de.}
\hspace{1ex}and Tobias Weth\footnote{Goethe-Universit\"at, Frankfurt, weth@math.uni-frankfurt.de.}
}
\date{\today}

\begin{document}
\maketitle
\begin{abstract}
We consider the nonlinear problem 
\[
 (P)\qquad\left\{\begin{aligned}
  I u&=f(x,u)&& \text{ in $\Omega$,}\\
u&=0&&\text{ on $\R^{N}\setminus\Omega$}\\
 \end{aligned}\right.
\]
in an open bounded set $\Omega\subset\R^{N}$, where $I$ is a nonlocal operator which may be anisotropic and may have varying order. We assume mild symmetry and monotonicity assumptions on $I$, $\Omega$ and the nonlinearity $f$ with respect to a fixed direction, say $x_1$, and we show that any nonnegative weak solution $u$ of $(P)$ is symmetric in $x_1$. Moreover, we have the following alternative: Either $u\equiv 0$ in $\Omega$, or $u$ is strictly decreasing in $|x_1|$. The proof relies on new maximum principles for antisymmetric supersolutions of an associated class of linear problems.
\end{abstract}

{\footnotesize
\begin{center}
\textit{Keywords.} Nonlocal Operators $\cdot$ Maximum Principles $\cdot$ Symmetries
\end{center}
\begin{center}
\end{center}
}
\section{Introduction}
In this work we study the following class of nonlocal and semilinear Dirichlet problems in a bounded open set $\Omega \subset \R^N$:  
\[
 (P)\qquad\left\{\begin{aligned}
  I u&=f(x,u)&& \text{ in $\Omega$;}\\
u&=0&&\text{ on $\R^{N}\setminus\Omega$.}
 \end{aligned}\right.
\]
Here the nonlinearity $f:\Omega\times \R\to\R$ is a measurable function with properties to be specified later, and $I$ is a nonlocal linear operator. Due to various applications in physics, biology and finance with anomalous diffusion phenomena,  nonlocal problems have gained enormous attention recently. In particular, problem $(P)$ has been studied with $I=(-\Delta)^{\frac{\alpha}{2}}$, the fractional Laplacian of order $\alpha \in (0,2)$. In this case, special properties of the fractional Laplacian have been used extensively to study existence, regularity and symmetry of solutions to $(P)$. In particular, some approaches rely on available Green function representions 
associated with $(-\Delta)^{\frac{\alpha}{2}}$, (see e.g. \cite{BB00,BLW05,CFY13,Chen_Li_Ou,FW13-2,Chen_Song}), whereas other techniques are based on a representation of $(-\Delta)^{\frac{\alpha}{2}}$ as a Dirichlet-to-Neumann map (see e.g \cite{CS07,CS,FLS13}).  These useful features of the fractional Laplacian are closely linked to its isotropy and its scaling laws. However, in the modeling of anisotropic diffusion phenomena and of processes which do not exhibit similar properties, it is necessary to study more general nonlocal operators $I$.  In this spirit, general classes of nonlocal operators have been considered e.g. in \cite{FK12,FKV13, SV}.\\
 In the present work  we consider $(P)$ for a class of nonlocal operators $I$ which includes the fractional Laplacian but also more general operators which may be anisotropic and may have varying order. More precisely,  the class of operators $I$ in $(P)$ is related to nonnegative nonlocal bilinear forms of the type 
 \begin{equation}
   \label{eq:def-cJ}
\cJ(u,v)=\frac{1}{2}\int_{\R^N}\int_{\R^{N}}(u(x)-u(y))(v(x)-v(y))J(x-y)\ dxdy
 \end{equation}
with a measurable function $J:\R^N  \setminus \{0\} \to[0,\infty)$.  We assume that $J$ is even, i.e, $J(-z)=J(z)$ for $z \in \R^{N} \setminus \{0\}$. Moreover, we assume the following integral condition:
$$
(J1) \qquad \quad \quad
\int_{\R^{N}\setminus B_{1}(0)} J(z)\ dz + \int_{B_{1}(0)} |z|^2 J(z)\ dz   < \infty \qquad \quad 
\text{and} \qquad \quad \int_{\R^{N}}J(z)\ dz = \infty. \qquad\quad
$$
By similar arguments as in the recent paper \cite{FKV13}, we shall see in Section~\ref{setup} below that this assumption ensures that $\cJ$ is closed and symmetric quadratic form in $L^2(\Omega)$ with a dense domain given by  
\begin{equation}
  \label{eq:def-cD}
\cD(\Omega):=\{\text{$u:\R^{N}\to\R$ measurable}\;:\; \cJ(u,u)<\infty \text{ and $u\equiv 0$ on $\R^{N}\setminus \Omega$}\}
\end{equation}
Here and in the following, we identify $L^2(\Omega)$ with the space of functions $u \in L^2(\R^N)$ with $u \equiv 0$ on $\R^N \setminus \Omega$. Consequently, $\cJ$ is the quadratic form of a unique self-adjoint operator $I$ on $L^2(\Omega)$, which also satisfies 
$$
[I u](x)= \lim_{\eps \to 0} \int_{|y-x|\ge \eps} [u(x)-u(y)]J(x-y)\,dy \qquad \text{for $u \in \cC^2_c(\Omega)$, $x \in \R^N$}
$$
see Corollary~\ref{3-dense} below. One may study solutions $u$ of $(P)$ in strong sense, requiring that $u$ is contained in the domain of the operator $I$. However, it is more natural to consider the weaker notion of solutions given by the quadratic form $\cJ$ itself. More precisely, we call a function $u\in \cD(\Omega)$ {\em a solution of $(P)$} if the integral $\int_{\Omega}f(x,u(x))\varphi(x)\ dx$ exists for all $\varphi\in \cD(\Omega)$ and 
\[
\cJ(u,\varphi)=\int_{\Omega}f(x,u(x))\varphi(x)\ dx \qquad \text{for all $\varphi\in \cD(\Omega)$,}
\]
We note that the fractional Laplacian $I:= (-\Delta)^{\alpha/2}$ corresponds to the kernel $J(z)= c_{N,\alpha} |z|^{-N-\alpha}$ with $c_{N,\alpha}=  \alpha(2-\alpha)\pi^{-N/2}2^{\alpha-2}\frac{\Gamma(\frac{N+\alpha}{2})}{\Gamma(2-\frac{\alpha}{2})}$.  Our paper is motivated by recent symmetry results for nonlinear equations involving the fractional Laplacian (see \cite{BLW05,CFY13,Chen_Li_Ou,FW13-2,JW13,sciunzi}). More precisely, we present a general approach, based on maximum principles for antisymmetric functions, to investigate symmetry properties of bounded nonnegative solutions of $(P)$ in bounded Steiner symmetric open sets $\Omega$. We claim that this approach is simpler and more general than the techniques applied in the papers cited above. 
In particular, it also applies to anisotropic operators and operators of variable order. To state our main symmetry result, we first introduce the following geometric assumptions on $J$ and the set $\Omega$.
\begin{enumerate}
\item[$(D)$] $\Omega\subset \R^N$ is an open bounded set which is Steiner symmetric in $x_1$, i.e. for every $x\in \Omega$ and $s \in [-1,1]$ we have $(sx_1,x_2,\dots,x_N) \in \Omega$.  
\item[$(J2)$] The kernel $J$ is strictly monotone in $x_1$, i.e. for all $z' \in\R^{N-1}$, $s,t \in \R$ with $|s| <|t|$ we have $J(s,z') > J(t,z')$.
\end{enumerate}
Note that $(J2)$ in particular implies that $J$ is positive on $\R^N \setminus \{0\}$.  We may now state our main symmetry result.

\begin{satz}\label{sec:goal}
 Let $(J1),(J2)$ and $(D)$ be satisfied, and assume that the nonlinearity $f$ has the following properties. 
\begin{enumerate}
\item[$(F1)$] $f: \Omega\times \R\to \R$, $(x,u)\mapsto f(x,u)$ is a Carath\'eodory function such that for every bounded set $K \subset \R$ there exists $L=L(K)>0$ with 
\[
 \sup \limits_{x\in\Omega} |f(x,u)-f(x,v)|\leq L |u-v|  \quad\text{ for $u,v \in K$.}
\]
 \item[$(F2)$] $f$ is symmetric and monotone in $x_1$, i.e. for every $u \in \R$, $x \in \Omega$ and $s \in
[-1,1]$ we have $f(s x_{1},x_2,\dots,x_N,u) \ge f(x,u)$.
\end{enumerate} 
Then every nonnegative solution $u \in L^\infty(\Omega) \cap \cD(\Omega)$ of $(P)$ is symmetric in $x_1$.
Moreover, either $u \equiv 0$ in $\R^{N}$, or $u$ is strictly decreasing in $|x_1|$ and therefore 
satisfies 
\begin{equation}
  \label{eq:positivity-thm-1-1}
\underset{K}{\essinf}\, u>0\qquad \text{for every compact set $K \subset \Omega$.}
\end{equation}
\end{satz}

\noindent Here and in the following, if $\Omega$ satisfies $(D)$ and $u: \Omega \to \R$ is measurable, we say that $u$ is\\[0.1cm]
$\bullet$  {\em symmetric in $x_1$} if $u(-x_{1},x')=u(x_{1},x')$ for almost every $x= (x_{1},x')\in \Omega$.\\[0.1cm]
$\bullet$  {\em strictly decreasing in $|x_1|$} if for every $\lambda \in \R \setminus \{0\}$ and every compact set $K \subset \{x \in \Omega\::\:  \frac{x_1}{\lambda} > 1\}$ we have 
$$
\underset{x \in K}{\essinf}\, \bigl[u(2\lambda-x_1, x_2,\dots,x_N) -u(x)\bigr] >0.
$$

\begin{bem}
\label{sec:ex-op}  
We wish to single out a particular class of operators satisfying $(J1)$ and $(J2)$. Let $\alpha,\beta \in (0,2)$, $c>0$ and consider a measurable map $k:(0,\infty)\to(0,\infty)$ such that
\[
\frac{\rho^{-N}}{c}  \le k(\rho)\le c\rho^{-N-\alpha}\quad \text{for $\rho \le 1\qquad$ and}\qquad 
k(\rho) \le c\rho^{-N-\beta}\quad\text{for $\rho >1$.}\\   
  \]
Suppose moreover that $k$ is strictly decreasing on $(0,\infty)$, and let 
$|\cdot|_\sharp$ denote a norm on $\R^N$ with the property that $|(s,z')|_\sharp<|(t,z')|_\sharp$ for every $s,t \in \R$ with $|s|<|t|$ and $z' \in \R^{N-1}$. Then the kernel 
$$
J: \R^N \setminus \{0\} \to \R,\qquad  J(z)= k(|z|_\sharp)
$$
satisfies $(J1)$ and $(J2)$. As remarked before, the case where $|\cdot|_{\sharp}=|\cdot|$ is the euclidean norm on $\R^N$ and $k(\rho)=  c_{N,\alpha} \rho^{-N-\alpha}$ corresponds to the fractional Laplacian $I=(-\Delta)^{\alpha/2}$. The class defined here also includes operators of order varying between $0$ and $\alpha \in (0,2)$. In particular, zero order operators are admissible. Moreover, the choice of non-euclidean norms $|\cdot|_{\sharp}$ leads to anisotropic operators. In particular, for $1 \le p < \infty$, the norm 
\begin{equation}
  \label{eq:norm-anisotropic}
|x|_\sharp  = |x|_p:= \bigl(\sum_{i=1}^N |x_i|^p\bigr)^{1/p} \qquad \text{for $x \in \R^N$}
\end{equation}
has the required properties. 
\end{bem}

As a direct consequence of Theorem \ref{sec:goal} we have the following. Here $e_j \in \R^N$ denotes the $j$-th coordinate vector for $j=1,\dots,N$.
\begin{cor}\label{cor-goal}
 Let $J(z)=k(|z|_p)$, where $k$ is as in Remark~\ref{sec:ex-op}, $1 \le p< \infty$ and $|\cdot|_p$ is given in (\ref{eq:norm-anisotropic}).\\[0.1cm] 
(i) Let $\Omega\subset\R^{N}$ be Steiner symmetric in $x_1,\dots,x_N$ , i.e., for every $x\in \Omega$, $j=1,\dots,N$ and $s \in [0,2]$ we have $x-  s x_j e_j \in \Omega$. Moreover, let $f$ fulfill $(F1)$ and be symmetric and monotone in $x_1,\dots,x_N$, i.e. for every $u \in \R$, $x \in \Omega$, $j=1,\dots,N$ and $s \in
[0,2]$ we have $f(x-s x_j e_j ,u) \ge f(x,u)$. Then every nonnegative solution $u\in L^\infty(\Omega) \cap \cD(\Omega)$ of $(P)$ is symmetric in $x_1,\dots,x_N$. Moreover, 
either $u \equiv 0$ in $\R^{N}$, or $u$ is strictly decreasing in $|x_1|,\dots,|x_N|$ and therefore satisfies (\ref{eq:positivity-thm-1-1}).\\[0.1cm]  
(ii) If $p=2$, $\Omega\subset\R^{N}$ is a ball centered in $0$ and $f$ fulfills $(F1)$, $(F2)$ and is radial in $x$ i.e. $f(x,u)=f(|x|e_1,u)$ for $x\in \Omega$, then every nonnegative solution $u\in L^\infty(\Omega) \cap \cD(\Omega)$ of $(P)$ is radially symmetric. 
Moreover, either $u \equiv 0$ in $\R^{N}$, or $u$ is strictly decreasing in $|x|$ and therefore satisfies~(\ref{eq:positivity-thm-1-1}).
\end{cor}
In the special case where $I=(-\Delta)^{\frac{\alpha}{2}}$, $\alpha\in(0,2)$, Theorem~\ref{sec:goal} has been obtained by the authors in \cite[Corollary 1.2]{JW13} as a corollary of result on asymptotic symmetry for the corresponding parabolic problem. While some of the parabolic estimates in \cite{JW13} are not available for the class of nonlocal operators considered here, we will be able to formulate elliptic counterparts of some of the tools from \cite{JW13} in the present setting.  Independently from our work \cite{JW13}, a weaker variant of Theorem~\ref{sec:goal} in the special case $I=(-\Delta)^{\frac{\alpha}{2}}$, restricted to strictly positive solutions, is proved  in the very recent preprint \cite[Theorem 1.2]{sciunzi}, where also related problems for the fractional Laplacian with singular local linear terms are considered.  Corollary \ref{cor-goal}(ii) for $I=(-\Delta)^{\frac{\alpha}{2}}$, $\alpha\in(0,2)$ has been proved first by Birkner, L\'opez-Mimbela and Wakolbinger \cite{BLW05} for $I=(-\Delta)^{\frac{\alpha}{2}}$ and a nonlinearity $f=f(u)$ which is nonnegative and increasing. In the very recent papers \cite{CFY13,FW13-2},  Corollary \ref{cor-goal}(ii) is proved for strictly positive solutions in the case $I=(-\Delta)^{\frac{\alpha}{2}}$ under different assumptions on $f$. The proofs in these papers rely on the explicit form of the Green function associated with $(-\Delta)^{\frac{\alpha}{2}}$ in balls.\\
In order to explain the difference between considering nonnegative or positive solutions,  we point out that the conclusion~(\ref{eq:positivity-thm-1-1}) can be seen as  a strong maximum principle for bounded solutions of $(P)$ in open sets satisfying $(D)$ which is {\em not} true for the corresponding Dirichlet problem 
\begin{equation}
  \label{class}
 \left\{\begin{aligned}
 -\Delta u&=f(x,u)&& \text{ in $\Omega$;}\\
u&=0&&\text{ on $\partial \Omega$.}
 \end{aligned}\right.
 \end{equation}
Note that we do not assume $\Omega$ to be connected in Theorem~\ref{sec:goal}, but even in domains $\Omega \subset \R^N$ the assumptions $(D)$ and $(F1)$, $(F2)$ do not guarantee that nonnegative solutions of (\ref{class}) are either strictly positive or identically zero in $\Omega$, see e.g. \cite{PT12} for examples for nonnegative solutions of (\ref{class}) with interior zeros. The positivity property (\ref{eq:positivity-thm-1-1}) can be seen as a consequence of the long range nonlocal interaction enforced by $(J2)$. Note that $(J2)$ is not satisfied for kernels of the form 
\begin{equation}
  \label{eq:cut-off-fractional}  
z \mapsto J(z)= 1_{B_r(0)}|z|^{-N-\alpha}\qquad \text{with $\alpha \in (0,2)$, $r>0$.}
\end{equation}
It is therefore natural to ask whether a result similar to Theorem~\ref{sec:goal} also holds for kernels of the type (\ref{eq:cut-off-fractional}) which vanish outside a compact set and therefore model short range nonlocal interaction. Related to this case, we have to following result for {\em a.e. positive solutions of $(P)$ in $\Omega$.}

\begin{satz}
\label{sec:vari-symm-result-1}
 Let $\Omega\subset\R^{N}$ satisfy $(D)$, and let the even kernel $J: \R^N \setminus \{0\} \to [0,\infty)$ satisfy $(J1)$ and 
\begin{enumerate}
\item[$(J2)'$] For all $z' \in\R^{N-1}$, $s,t \in \R$ with $|s| \le |t|$ we have $J(s,z') \ge J(t,z')$.
Moreover, there is $r_0>0$ such that 
$$J(s,z') > J(t,z') \qquad \text{for all $z' \in\R^{N-1}$ and $s,t \in \R$, with $|z'| \le r_0$ and $|s| <|t|\le r_0$.}   
$$
\end{enumerate}
Furthermore, suppose that the nonlinearity satisfies $(F1)$ and $(F2)$. Then every a.e. positive solution $u\in L^\infty(\Omega) \cap \cD(\Omega)$ of $(P)$ is symmetric in $x_1$ and strictly decreasing in $|x_1|$ on $\Omega$. Consequently, it satisfies~(\ref{eq:positivity-thm-1-1}).
\end{satz}
Note that the kernel class given by (\ref{eq:cut-off-fractional}) satisfies $(J1)$ and $(J2)'$. We recall that Gidas, Ni and Nirenberg \cite{GNN79} proved the corresponding symmetry result for strictly positive solutions of (\ref{class}) under some restrictions on $\Omega$ which were then removed in \cite{BN91}. These results rely on the moving plane method which, in other variants, had already been introduced in \cite{A62,S71}. For nonlocal problems involving the fractional Laplacian, the moving plane method was used in a stochastic framework by Birkner, L\'opez-Mimbela and Wakolbinger in the above-mentioned paper \cite{BLW05}. Chen, Li and Ou \cite{Chen_Li_Ou} used the explicit form of the inverse of the fractional Laplacian to prove symmetry results for $I=(-\Delta)^{\frac{\alpha}{2}}$ and $f(u)=u^{(N+\alpha)/(N-\alpha)}$ in $\R^{N}$. For this they developed a variant of the moving plane method for integral equations. Similar methods were used in the above-mentioned papers \cite{CFY13,FW13-2}.\\
The results on the present paper rely on a different variant of the moving plane method which partly extends recent techniques of \cite{JW13, FJ13,RS13} and, independently, \cite{sciunzi}. More precisely, we show that $(J1)$ and $(J2)$ -- or, alternatively, $(J2)'$ -- are sufficient assumptions for the bilinear form $\cJ$ to provide maximum principles for antisymmetric solutions of associated linear operator inequalities in weak form, see Section~\ref{mp}. Here antisymmetry refers to a reflection at a given hyperplane.  Combining different (weak and strong) versions of these maximum principles, we then develop a framework for the moving plane method for nonnegative solutions of $(P)$ which are not necessarily strictly positive. The approach seems more direct and more flexible than the ones in \cite{CFY13,Chen_Li_Ou,FW13-2} since it does not depend on Green function representations.\\ 
The paper is organized as follows. In Section \ref{setup} we collect useful properties of the nonlocal bilinear forms which we consider. Section \ref{mp} is devoted to classes of linear problems related to $(P)$ and hyperplane reflections. In particular, we prove a small volume type maximum principle and a strong maximum principle for antisymmetric supersolutions of these problems. In Section~\ref{mr} we complete the proof of Theorem~\ref{sec:goal}, and in Section~\ref{sec:vari-symm-result} we complete the proof of Theorem~\ref{sec:vari-symm-result-1}.\\


\textbf{Acknowledgment:} 
Part of this work was done while the
first author was visiting AIMS-Senegal. He would like to thank them for their kind hospitality.

\section{Preliminaries}\label{setup}
We fix some notation. For subsets $D,U \subset \R^N$ we write $\dist(D,U):= \inf\{|x-y|\::\: x \in D,\, y \in U\}$.  If $D= \{x\}$ is a singleton, we write $\dist(x,U)$
in place of $\dist(\{x\},U)$. For $U\subset\R^{N}$ and $r>0$ we consider $B_{r}(U):=\{x\in\R^{N}\;:\; \dist(x,U)<r\}$, and we let, as usual  
 $B_r(x)=B_{r}(\{x\})$ be the open ball in $\R^{N}$ centered at $x \in \R^N$ with radius $r>0$. For any subset $M \subset \R^N$, we denote by $1_M: \R^N \to \R$ the
characteristic function of $M$ and by $\diam(M)$ the diameter of $M$. If $M$ is measurable $|M|$ denotes the Lebesgue measure of $M$. Moreover, if $w: M \to \R$ is a function,  we let $w^+= \max\{w,0\}$ resp. $w^-=-\min\{w,0\}$ denote the positive and negative part of $w$, respectively. 

Throughout the remainder of the paper, we assume that $J:\R^N \setminus \{0\}\to[0,\infty)$ is even and satisfies $(J1)$. We let $\cJ$ be the corresponding quadratic form defined in (\ref{eq:def-cJ}) and, for an open set $\Omega \subset \R$, we consider $\cD(\Omega)$ as defined in (\ref{eq:def-cD}). It follows from $(J1)$ that $J$ is positive on a set of positive measure.  Thus, by   
\cite[Lemma 2.7]{FKV13} we have $\cD(\Omega) \subset L^2(\Omega)$ and 
\begin{equation}
\label{l2-bound}
\Lambda_{1}(\Omega):=\inf_{u\in \cD(\Omega)}\frac{\cJ(u,u)}{\|u\|^2_{L^{2}(\Omega)}} \: >\:0 \qquad \text{for every open bounded set  $\Omega\subset\R^{N}$,}
\end{equation}
which amounts to a Poincar\'e-Friedrichs type inequality. We will need lower bounds for $\Lambda_1(\Omega)$ in the case where $|\Omega|$ is small. For this we set
$$
\Lambda_1(r):= \inf \{ \Lambda_{1}(\Omega)\::\: \text{$\Omega \subset \R^N$ open, $|\Omega|=r$}\} \qquad \text{for $r>0$.}
$$

\begin{lemma}\label{3-mengen-k}
We have $\Lambda_{1}(r) \to \infty$ as $r \to 0$.
\end{lemma}

\begin{proof}
Let  
$$
 J_c:= \{ z \in \R^N \setminus \{0\}\::\: J(z) \ge c\} \qquad \text{and}\qquad J^c:= \{ z \in \R^N \setminus \{0\}\::\: J(z) < c\}
$$
for $c \in [0,\infty]$. We also consider the decreasing rearrangement $d:(0,\infty) \to [0,\infty]$ of $J$ given by $d(r)= \sup \{c \ge 0 \::\: |J_c| \ge r \}$. We first note that 
\begin{equation}
 \label{eq:est-dec-rearr-3}
 |J_{d(r)}| \ge r \qquad \text{for every $r>0$} 
\end{equation}
Indeed, this is obvious if $d(r)=0$, since $J_0= \R^N \setminus \{0\}$. If $d(r)>0$, we have $|J_c| \ge r$ for every $c < d(r)$ by definition, whereas $|J_c| < \infty$ for every $c>0$ as a 
consequence of the fact that $J \in L^{1}(\R^{N}\setminus B_1(0))$ by $(J1)$. Consequently, since $J_{d(r)}= \underset{c < d(r)}{\bigcap} J_c$, we have $|J_{d(r)}| =   \inf \limits_{c < d(r)} |J_c| \ge r.$ Next we claim that 
\begin{equation}
  \label{eq:est-dec-rearr}
\Lambda_1(r)  \ge \int_{J^{d(r)}} J(z)\,dz \qquad \text{for $r>0$.}
\end{equation}
Indeed, let $r>0$ and $\Omega \subset \R^N$ be measurable with $|\Omega|=r$. 
For $u\in \cD(\Omega)$ we have
\begin{align}
\cJ(u,u)&=\frac{1}{2}\int_{\R^{N}}\int_{\R^{N}}(u(x)-u(y))^2J(x-y)\ dxdy \nonumber \\
&=\frac{1}{2}\int_{\Omega}\int_{\Omega}(u(x)-u(y))^2 J(x-y)\ dxdy+\int_{\Omega}u^2(x)\int_{\R^{N}\setminus \Omega} J(x-y)\ dy \ dx  \nonumber\\
&\geq \inf_{x\in \Omega}\biggl(\;\int_{\;\R^{N}\setminus \Omega_x} J(y)\ dy\biggr)\|u\|^2_{L^{2}(\Omega)} \label{eq:est-dec-rearr-4}
\end{align}
with $\Omega_x:=x+\Omega$. Let $d:=d(r)$. Since $|J_{d}| \ge r = |\Omega|$ by (\ref{eq:est-dec-rearr-3}), we have $|J_d \setminus \Omega_x| \ge |\Omega_x \setminus J_d|$ and thus, for every $x \in \Omega$, 
\begin{align*}
 \int_{\R^{N}\setminus \Omega_x} J(y)\ dy&=\int_{\R^N \setminus J_d} J(y)\ dy +\int_{J_d\setminus \Omega_x}J(y)\ dy- \int_{\Omega_{x}\setminus J_d}J(y)\ dy\\
&\ge \int_{J^d} J(y)\ dy + \Bigl(|J_d\setminus \Omega_x| - 
|\Omega_{x}\setminus J_d|\Bigr)d  \geq \int_{J^d} J(y)\ dy.
\end{align*}
Combining this with (\ref{eq:est-dec-rearr-4}), we obtain (\ref{eq:est-dec-rearr}), as  claimed. As a consequence of the second property in $(J1)$, the decreasing rearrangement of $J$ satisfies $d(r) \to \infty$ as $r \to 0$ and  
$$
\int_{J^{d(r)}} J(y)\ dy \to \infty \qquad \text{as $r \to 0$.}
$$
Together with (\ref{eq:est-dec-rearr}), this shows the claim.  
\end{proof}

\begin{prop}
\label{complete}
Let $\Omega \subset \R^N$ be open and bounded. Then $\cD(\Omega)$ is a Hilbert space with the scalar product $\cJ$.  
 \end{prop}

 \begin{proof}
We argue similarly as in the proof of \cite[Lemma 2.3]{FKV13}. Let $(u_n)_n\subset\cD(\Omega)$ be a Cauchy sequence. By (\ref{l2-bound}) and the completeness of $L^2(\Omega)$, we have that $u_n\to u\in L^{2}(\Omega)$ for a function $u\in L^{2}(\Omega)$. Hence there exists a subsequence such that $u_{n_k}\to u$ almost everywhere in $\Omega$ as $k \to \infty$. By Fatou's Lemma, we therefore have that 
\[
 \cJ(u,u)\leq \liminf_{k\to\infty}\cJ(u_{n_k},u_{n_k})\leq \sup_{k\in \N}\cJ(u_{n_k},u_{n_k})<\infty,
\]
so that $u \in \cD(\Omega)$. Applying Fatou's Lemma again, we find that 
$$
\cJ(u_{n_k}-u,u_{n_k}-u) \le \liminf_{j \to \infty}  \cJ(u_{n_k}-u_{n_j},u_{n_k}-u_{u_j}) \le \sup_{j \ge k} \cJ(u_{n_k}-u_{n_j},u_{n_k}-u_{u_j}) \qquad \text{for $k \in \N$}.
$$
Since $(u_{n})_n$ is a Cauchy sequence with respect to the scalar product $\cJ$, it thus follows that $\lim \limits_{k \to \infty}u_{n_k} = u$ and therefore also $\lim \limits_{n \to \infty}u_{n} = u$ in $\cD(\Omega)$. This shows the completeness of $\cD(\Omega)$.
 \end{proof}

\begin{prop}\label{3-prel-dense}
(i) We have $\cC^{0,1}_{c}(\R^{N}) \subset \cD(\R^N)$.\\[0.1cm]
(ii) Let $v \in \cC_c^2(\R^N)$. Then the principle value integral 
\begin{equation}
  \label{eq:princ-value}
[Iv](x):= P.V. \int_{\R^N} (v(x)-v(y)) J(x-y)\,dy = \lim_{\eps \to 0} \int_{|x-y|\ge \eps} (v(x)-v(y)) J(x-y)\,dy
\end{equation}
exists for every $x \in \R^N$. Moreover, $I v \in L^\infty(\R^N)$,  and for every bounded open set $\Omega \subset \R^N$ and every $u \in \cD(\Omega)$ we have 
$$
\cJ(u,v)= \int_{\R^N} u(x) [Iv](x)\,dx.
$$
  
\end{prop}

\begin{proof}
(i) Let $u\in \cC^{0,1}_{c}(\R^{N})$, and let $K>0$, $R>2$ be such that $\supp(u) \subset B_{R-2}(0)$, 
$$
|u(x)|\leq K \quad \text{and}\quad |u(x)-u(y)|\leq K|x-y| \qquad \text{for $x,y \in \R^N$, $x \not=y$.}
$$
Then, as a consequence of $(J1)$, 
\begin{align*}
2\cJ(u,u)&=\int_{ B_{R}(0)}\int_{ B_{R}(0)}(u(x)-u(y))^2 J(x-y)\ dxdy +2\int_{ B_{R}(0)}u^{2}(x)\int_{\R^{N}\setminus  B_{R}(0)}J(x-y)\ dydx\\
&\leq K^2 \int_{ B_{R}(0)}\int_{ B_{R}(0)}|x-y|^{2} J(x-y)\ dx dy +2K^2  \int_{ B_{R-2}(0)}\;\int_{\R^{N}\setminus  B_{R}(0)}J(x-y)\ dydx\\
&\leq 2 K^2|B_R(0)|\Bigl( \int_{ B_{2R}(0)}|z|^{2} J(z)\ dz  + \int_{\R^{N}\setminus  B_{1}(0)}J(z)\ dz \Bigr)<\infty 
\end{align*}
and thus $u \in \cD(\R^N)$.\\
(ii)  Since $v\in \cC^{2}_{c}(\R^{N})$, there exist constants $\delta,K>0$ such that 
\begin{equation}
  \label{eq:C2-local-est}
|2v(x)-v(x+z)-v(x-z)|\leq K |z|^2 \qquad \text{for all $x,z \in \R^N$ with $|z| \le \delta$.}
\end{equation}
Put $h(x,y):= (v(x)-v(y))J(x-y)$ for $x,y \in \R^N$, $x \not = y$.
For every $x \in \R^N$, $\eps \in (0,\delta)$ we then have, since $J$ is even,
\begin{align*}
\int_{\eps \le |y-x|\le \delta} h(x,y)\,dy &= \int_{\eps \le |z| \le \delta} [v(x)-v(x+z)] J(z)\,dz= \int_{\eps \le |z| \le \delta} [v(x)-v(x-z)] J(z)\,dz\\
& = \frac{1}{2} \int_{\eps \le |z| \le \delta} [2 v(x)-v(x+z)-v(x-z)]J(z)\,dz.
\end{align*}
By the first inequality in $(J1)$, (\ref{eq:C2-local-est}) 
and Lebesgue's theorem we thus conclude the existence of the limit 
$$
\lim_{\eps \to 0} \int_{\eps \le |y-x|\le \delta}h(x,y)\,dy =\frac{1}{2} \int_{0 \le |z| \le \delta} [2 v(x)-v(x+z)-v(x-z)]J(z)\,dz.
$$
Moreover we have for $x \in \R^N$ and $\eps \in (0,\delta)$ 
\begin{equation}
  \label{eq:extra-pf-3-prel-dense}
\int_{|y-x| \ge \eps}  h(x,y) \,dy \le 2 \|v\|_{L^\infty(\R^N)} \int_{\R^N \setminus B_\delta(0)}J(z)\,dz + \frac{K}{2}\int_{B_\delta(0)} |z|^2 J(z)\,dz =:K', 
\end{equation}
where the right hand side is finite by the first inequality in $(J1)$. In particular, $[Iv](x)$ is well defined by (\ref{eq:princ-value}), and $|[Iv](x)\bigr |  \le K'$ for $x \in \R^N$, so that $Iv \in L^\infty(\R^N)$. Next, let $\Omega \subset \R^N$ be open and bounded and $u\in \cD(\Omega)$, so that also $u \in L^2(\Omega)$. Then we have, by (\ref{eq:extra-pf-3-prel-dense}) and Lebesgue's Theorem, 
\begin{align*}
\cJ&(u,v)= \frac{1}{2} \lim_{\eps \to 0} \int_{|x-y|\ge \eps} (u(x)-u(y))h(x,y) \,dx \,dy\\
&= \lim_{\eps \to 0} \int_{\R^N} u(x) \int_{|y-x| \ge \eps} h(x,y) \,dy dx = \int_{\R^N} u(x) 
\Bigl[\: \lim_{\eps \to 0} \int_{|y-x| \ge \eps} h(x,y) \,dy\Bigr] dx
=\int_{\R^N} u(x) [Iv](x)\,dx.
\end{align*}

The proof is finished.
\end{proof}

\begin{cor}\label{3-dense}
Let $\Omega\subset \R^{N}$ be open and bounded. Then $\cJ$ is a closed quadratic form with dense form domain $\cD(\Omega)$ in $L^{2}(\Omega)$. Consequently, $\cJ$ is the quadratic form of a unique self-adjoint operator $I$ in $L^2(\Omega)$. Moreover, $C_c^2(\Omega)$ is contained in the domain of $I$, and for every $v \in \cC_c^2(\Omega)$ the function 
$Iv \in  L^2(\Omega)$ is a.e. given by (\ref{eq:princ-value}).  
\end{cor}

\begin{proof}
Since $\cC^{0,1}_{c}(\Omega) \subset L^2(\Omega)$ is dense, $\cD(\Omega)$ is a dense subset of $L^2(\Omega)$ by Proposition \ref{3-prel-dense}(i). Moreover, the quadratic form $\cJ$ is closed in $L^2(\Omega)$ as a consequence of (\ref{l2-bound}) and Lemma~\ref{complete}. Hence $\cJ$ is the quadratic form of a unique self-adjoint operator $I$ in $L^2(\Omega)$ (see e.g. \cite[Theorem VIII.15, pp. 278]{SR}). Moreover, for every $v \in \cC_c^2(\Omega)$, $u \in \cD(\Omega)$ we have $|J(u,v)| \le |\Omega| \|Iv\|_{L^\infty(\Omega)} \|u\|_{L^2(\Omega)}$ by Proposition~\ref{3-prel-dense}(ii). Consequently, $v$ is contained in the domain of $I$ and satisfies $J(u,v) = \int_{\R^N}u [Iv]\,dx$ for every $u \in \cD(\Omega)$. From Proposition~\ref{3-prel-dense}(ii) it then follows that $Iv$ is a.e. given by (\ref{eq:princ-value}).   
\end{proof}

Next, we wish to extend the definition of $\cJ(v,\phi)$ to more general pairs of functions $(v,\phi)$. In the following, for a measurable subset $U' \subset \R^N$, we define $\cH(U')$ as the space of all functions $v \in L^2(\R^N)$ such that 
\begin{equation}\label{subset}
\rho(v,U'):=  \int_{U'}\int_{U'}(v(x)-v(y))^2J(x-y)\ dxdy<\infty.
\end{equation}
Note that $\cD(\R^N) \cap L^2(\R^N) \subset \cH(U')$ for any measurable subset $U' \subset \R^N$, and thus also $\cD(U) \subset \cH(U')$ for any bounded open set $U \subset \R^N$ by (\ref{l2-bound}).

\begin{lemma}
\label{sec:linear-problem-tech-1}  
Let $U' \subset \R^N$ be an open set and $v,\phi \in \cH(U')$. Moreover, suppose that $\phi \equiv 0$ on $\R^N \setminus U$ for some subset $U \subset U'$ with $\dist(U,\R^N \setminus U')>0$. Then 
\begin{equation}
  \label{eq:finite-int}
\int_{\R^N} \int_{\R^N} |v(x)-v(y)| |\phi(x)-\phi(y)|J(x-y) \,dx dy < \infty,  
\end{equation}
and thus 
$$
\cJ(v,\phi) := \frac{1}{2} \int_{\R^N} \int_{\R^N} (v(x)-v(y)) (\phi(x)-\phi(y))J(x-y) \,dx dy  
$$
is well defined.   
\end{lemma}
 
\begin{proof}
Since $J$ satisfies $(J1)$, we have $K:= \int_{\text{\tiny $\R^N \setminus B_r(0)$}}J(z)\,dz < \infty
$ with $r:= \dist(U,\R^N \setminus U')>0$. As a consequence, 
\begin{align*}
\int_{\R^N}&  \int_{\R^N} |v(x)-v(y)| |\phi(x)-\phi(y)|J(x-y) \,dx dy\\
 &= \int_{U'}\int_{U'}|v(x)-v(y)||\varphi(x)-\varphi(y)|J(x-y)\ dxdy + 2 \int_{U}\int_{\R^{N}\setminus U'} |v(x)-v(y)| |\varphi(x)|J(x-y)\ dydx\\
& \le \frac{1}{2} \bigl[\rho(v,U') + \rho(\phi,U')\bigr] +  \int_{U}\int_{\R^{N}\setminus U'} \Bigl[2 \bigl(|v(x)|^2+ |v(y)|^2\bigr) + |\varphi(x)|^2\Bigr] J(x-y)\ dydx\\
& \le \frac{1}{2} \bigl[\rho(v,U') + \rho(\phi,U')\bigr] +  K \Bigl( 4 \|v\|_{L^2(\R^N)}^2+ \|\phi\|_{L^2(\R^N)}^2\Bigr)<\infty.
\end{align*}
 \end{proof}

\begin{lemma}\label{3-cutoff}
If $U' \subset \R^N$ is open and $v\in \cH(U')$, then $v^\pm \in \cH(U')$ and $\rho(v^\pm,U') \le \rho(v,U')$.
\end{lemma}
\begin{proof}
We have $v^\pm \in L^2(\R^N)$ since $v \in L^2(\R^N)$. Moreover, $v^+(x) v^-(x) =0$ for $x \in \R^N$ and thus
\begin{align*}
&\rho(v,U')= \rho(v^+,U') + \rho(v^-,U') - 2\int_{U'}\int_{U'} (v^+(x)-v^+(y))(v^-(x)-v^-(y))J(x-y)\ dxdy\\  
&=\rho(v^+,U') + \rho(v^-,U') +2 \int_{U'}\int_{U'}[v^{+}(x)v^{-}(y)+v^{+}(y)v^{-}(x)] J(x-y)\ dxdy\\
&\geq \rho(v^+,U') + \rho(v^-,U'). 
\end{align*}
The claim follows.
\end{proof}

We close this section with a remark on assumption $(J2)$.

\begin{bem}
\label{sec:equality-monotonicity}
Suppose that $(J2)$ is satisfied. Then, for every fixed $z' \in \R^N$, the function $t \mapsto J(t,z')$ is strictly decreasing in $|t|$ and therefore coincides a.e. on $\R$ with 
the function $t \mapsto \tilde J(t,z'):= \lim \limits_{s \to t^-}J(s,z')$. Hence $J$ and the function $\tilde J$ differ only on a set of measure zero in $\R^N$. Replacing $J$ by $\tilde J$ if necessary, we may therefore deduce from $(J2)$ the symmetry property  
\begin{equation}
  \label{eq:adj-measure}
J(-t,z')= J(t,z') \qquad \text{for every $z' \in \R^{N-1}, t\in \R$.}   
\end{equation}
This will be used in the following section.   
\end{bem}

\section{The linear problem associated with a hyperplane reflection}\label{mp}

In the following, we consider a fixed open affine half space $H \subset \R^N$, and we let $Q: \R^N \to \R^N$ denote the reflection at $\partial H$. For the sake of brevity, we sometimes write $\bar x$ in place of $Q(x)$ for $x \in \R^N$. A function $v:\R^{N}\to \R^{N}$ is called  antisymmetric (with respect to $Q$) if $v(\bar x)=-v(x)$ for $x \in \R^{N}$. As before, we consider an even kernel $J: \R^N \setminus \{0\} \to [0,\infty)$ satisfying $(J1)$. We also assume the following symmetry and monotonicity assumptions on $J$: 
\begin{align}
&J(\bar x-\bar y)=J(x-y) \qquad \text{for all $x,y \in \R^N$;} \label{sym-Q-J-1} \\
&J(x-y) \ge J(x- \bar y) \qquad \text{for all $x,y \in H$.}  \label{sym-Q-J-2}
\end{align}
 
\begin{bem}\label{symmetry-need}
If $(J1)$, $(J2)$ and (\ref{eq:adj-measure}) are satisfied and 
$$
H= \{x \in \R^N\::\: x_1 > \lambda\} \qquad \text{or}\qquad H= \{x \in \R^N\::\: x_1 < -\lambda\}
$$
for some $\lambda \ge 0$, then (\ref{sym-Q-J-1}) and (\ref{sym-Q-J-2}) hold. If $\lambda>0$, then $J$ also satisfies the following strict variant of (\ref{sym-Q-J-2}):
\begin{equation}
  \label{sym-Q-J-2-strict}
J(x-y) > J(x- \bar y) \qquad \text{for all $x,y \in H$.}  
\end{equation}
We will need this property in Proposition~\ref{hopf-simple2} below.
 \end{bem}

\begin{lemma}
\label{sec:linear-problem-tech}  
Let $J$ satisfy $(J1)$, (\ref{sym-Q-J-1}) and (\ref{sym-Q-J-2}). Moreover, let $U' \subset \R^N$ be an open set with $Q(U')=U'$, and let $v \in \cH(U')$ be an antisymmetric function such that $v \ge 0$ on $H \setminus U$ for some open bounded set $U \subset H$ with 
$\overline U \subset U'$.  Then the function $w:= 1_H\, v^-$ is contained in $\cD(U)$ and satisfies 
\begin{equation}
  \label{eq:key-ineq}
\cJ(w,w) \le - \cJ(v,w) 
\end{equation}
\end{lemma}

\begin{proof}
We first show that $w \in \cH(U')$. Clearly we have $w \in L^2(\R^N)$, since $v \in L^2(\R^N)$. Moreover, by (\ref{sym-Q-J-1}), the symmetry of $U'$, the antisymmetry of $v$ and (\ref{sym-Q-J-2}) we have 
\begin{align}
 &\rho(v,U') =\int_{U'\cap H}\int_{U'\cap H}(v(x)-v(y))^2 J(x-y)\ dxdy  \notag\\
&\qquad\qquad+\int_{U'\setminus H}\int_{U'\setminus H}(v(x)-v(y))^2 J(x-y)\ dxdy+2\int_{U'\setminus H}\int_{U'\cap H}(v(x)-v(y))^2 J(x-y)\ dxdy  \notag\\
&=2\int_{U'\cap H}\int_{U'\cap H}\Bigl[(v(x)-v(y))^2 J(x-y) + (v(x)+v(y))^2 J(x-\bar y)\Bigr]\ dxdy  \notag\\
&\geq \int_{U'\cap H}\int_{U'\cap H} \Bigl[(v(x)-v(y))^2 J(x-y) + [(v(x)-v(y))^2 + (v(x)+v(y))^2] J(x-\bar y)\Bigr]\ dxdy  \notag\\ 
&\geq \int_{U'\cap H}\int_{U'\cap H} \Bigl[(v(x)-v(y))^2 J(x-y) + 2v^2(x) J(x-\bar y)\Bigr]\ dxdy  \notag\\ 
&= \int_{U'}\int_{U'}(1_{H}v(x)-1_{H}v(y))^2 J(x-y)\ dxdy = \rho(1_{H}\,v ,U') \label{bound3}
\end{align}
and thus $\rho(1_{H}\,v,U') < \infty$. Hence $1_{H}\,v \in \cH(U')$ and thus also $w \in \cH(U')$ by Lemma~\ref{3-cutoff}. Since $w \equiv 0$ in $\R^N \setminus U$, the right hand side of (\ref{eq:key-ineq}) is well defined and finite by Lemma~\ref{sec:linear-problem-tech-1}.  To show  (\ref{eq:key-ineq}), we first note that 
$$
[w+v]w= [ 1_H v^+ + 1_{\R^N \setminus H} v]1_H v^- \equiv 0 \qquad  \text{on $\R^N$}
$$
and therefore  
$$
[w(x)-w(y)]^2 + [v(x)-v(y)] [w(x)-w(y)] = - \Bigl(w(x)[w(y)+v(y)] + w(y)[w(x)+v(x)]\Bigr)
$$
for $x,y \in \R^N$. Using this identity in the following together with the antisymmetry of $v$, the symmetry properties of $J$ and the fact that $w \equiv 0$ on $\R^N \setminus H$, we find that 
\begin{align*}
\cJ(w,w) + \cJ(v,w) &= 
-\int_{H} \int_{\R^N}w(x) [w(y)+v(y)] J(x-y)\,dy dx\\
 &=-\int_{H} \int_{\R^N}w(x) [1_H(y)v^+(y)  +1_{\R^N \setminus H} v(y)] J(x-y)\,dy dx\\
&=-\int_{H} \int_{H}w(x) [v^+(y)J(x-y) -v(y) J(x-\bar y)] \,dy dx \:\le\: 0, 
 \end{align*}
where in the last step we used the fact that $v^+(y) \ge v(y)$ and $J(x-y) \ge J(x-\bar y) \ge 0$ for $x,y \in H$. Hence (\ref{eq:key-ineq}) is true, and in particular we have $\cJ(w,w)< \infty$. Since $w \equiv 0$ on $\R^N \setminus U$, it thus follows that $w \in \cD(U)$.
\end{proof}

In order to implement the moving plane method, we have to deal with the class of antisymmetric supersolutions of a class of linear problems. A related notion was introduced in \cite{JW13} in a parabolic setting related to the fractional Laplacian. 

\begin{defi}\label{3-defi-anti}
Let $U\subset H$ be an open bounded set and let $c\in L^{\infty}(U)$.  We call an antisymmetric 
function $v: \R^N \to \R^N$ an \textit{antisymmetric supersolution} of the problem 
\begin{equation}\label{linear-prob}
Iv= c(x)v\quad \text{ in $U$,}\qquad  v \equiv 0 \quad \text{on $H \setminus U$}
\end{equation}
if $v \in \cH(U')$ for some open bounded set $U' \subset \R^N$ with $Q(U')=U'$ and $\overline U \subset U'$,  $v\geq 0$ on $H\setminus U$ and 
\begin{equation}\label{3-eq-sol2}
\cJ(u,\phi)\geq \int_{U}c(x)u(x)\varphi(x)\ dx \qquad \text{ for all $\varphi\in \cD(U)$, $\varphi\geq0$.} 
\end{equation}
\end{defi}

\begin{bem}\label{3-anti}
Assume $(J1)$ and (\ref{sym-Q-J-1}), and let $\Omega \subset \R^{N}$ be an open bounded set such that $Q(\Omega \cap H)\subset \Omega$. 
 Furthermore, let $f: \Omega \times\R\to\R$ be a Carath\'eodory function satisfying $(F1)$ and 
such that
 \begin{equation}
   \label{eq:F2H}
 f(\bar x,\tau) \geq f(x,\tau)\qquad \text{for every $\tau \in \R$, $x \in H\cap \Omega$.}  
 \end{equation}
If $u \in \cD(\Omega)$ is a nonnegative solution of $(P)$, then $v:=u\circ Q-u$ is an antisymmetric supersolution of (\ref{linear-prob}) with $U:= \Omega \cap H$ and $c \in L^\infty(U)$ defined by  
$$
c(x)= \left \{
  \begin{aligned}
  &\frac{f(x,u(\bar x))-f(x,u(x))}{v(x)}&&\qquad \text{if $v(x) \not= 0$;}\\
  &0 &&\qquad \text{if $v(x)= 0$.}
  \end{aligned}
\right.
$$
Indeed, since $u\in \cD(\Omega)$, we have $v \in \cD(\R^{N}) \cap L^2(\R^N)$ and thus $v \in 
\cH(U')$ for any open set $U' \subset \R^N$. Moreover, $v \ge 0$ on $H \setminus U$ since $u$ is nonnegative and $u \equiv 0$ on $H \setminus U$. Furthermore, if $\varphi\in \cD(U)$, then $\varphi\circ Q-\varphi\in \cD(\Omega)$ by the symmetry properties of $J$ and since $Q(U)\subset \Omega$. If, in addition, $\varphi \ge 0$, then we have, using (\ref{sym-Q-J-1}), 
\begin{align*}
&\cJ(v,\phi)= \cJ(u \circ Q - u,\phi)=\cJ(u,\phi \circ Q - \phi)= 
\int_{\Omega}f(x,u)[\varphi \circ Q - \varphi]\,dx\\
&=\!\!\!\int_{ Q(U)}\!\!\!f(x,u(x))\varphi \circ Q\,dx - \int_{U}\!\!f(x,u(x))\varphi\,dx=\!\!\int_{U}[f(\bar x,u(\bar x))-f(x,u(x))]\varphi(x)\,dx \ge \int_{U}c(x) v \varphi\,dx.
\end{align*}
Here (\ref{eq:F2H}) was used in the last step. The boundedness of $c$ follows from $(F1)$.
\end{bem} 

We now have all the tools to establish maximum principles for antisymmetric supersolutions of (\ref{linear-prob}). 

\begin{prop}\label{4-elliptic-max1}
Assume that $J$ satisfies $(J1)$, (\ref{sym-Q-J-1}) and (\ref{sym-Q-J-2}), and let $U \subset H$ be  an open bounded set.   Let $c\in L^{\infty}(U)$ with $\|c^+\|_{L^\infty(U)} <\Lambda_1(U)$, where $\Lambda_1(U)$ is given in (\ref{l2-bound}).\\
Then every antisymmetric supersolution $v$ of (\ref{linear-prob}) in $U$ satisfies $v\geq 0$ a.e. in $H$. 
\end{prop}
\begin{proof}
By Lemma \ref{sec:linear-problem-tech} we have that  
$w:=1_{H} v^{-} \in \cD(U)$ and $\cJ(w,w) \le -\cJ(v,w)$. Consequently,
\begin{align*}
\Lambda_{1}(U) \|w\|_{L^{2}(U)}^{2} \le \cJ(w,w) \le -\cJ(v,w) \le- \int_{U}c(x) v(x)w(x) \ dx
&= \int_{U}c(x) w^2(x) \ dx \\
&\le \|c^+\|_{L^\infty(U)}\|w\|_{L^{2}(U)}^{2}.
\end{align*}
Since $\|c^+\|_{L^\infty(U)}< \Lambda_{1}(U)$ by assumption, we conclude that $\|w\|_{L^{2}(U)}=0$ and hence $v\geq0$ a.e. in $H$.
\end{proof}

We note that a combination of Proposition~\ref{4-elliptic-max1} with Lemma~\ref{3-mengen-k} gives rise to an ``antisymmetric'' small volume maximum principle which generalizes the available variants for the fractional Laplacian, see \cite[Proposition 3.3 and Corollary 3.4]{FJ13} and \cite[Lemma 5.1]{RS13}. Next we prove a strong maximum principle which requires the strict inequality (\ref{sym-Q-J-2-strict}).  

\begin{prop}\label{hopf-simple2}
Assume that $J$ satisfies $(J1)$, (\ref{sym-Q-J-1}) and (\ref{sym-Q-J-2-strict}). Moreover, let $U \subset H$ be  an open bounded set and $c \in L^\infty(U)$. 
Furthermore, let  $v$  be an antisymmetric supersolution of (\ref{linear-prob})  
such that $v\geq0$ a.e. in $H$. Then either $v \equiv 0$ a.e. in $\R^{N}$, or 
$$
\underset{K}\essinf\: v >0 \qquad \text{for every compact subset $K  \subset U$.}
$$
\end{prop}

\begin{proof}
We assume that $v \not \equiv 0$  in $\R^N$. For given $x_0 \in U$, it then suffices to show that $\underset{B_r(x_0)}\essinf\: v >0$ for $r>0$ sufficiently small. Since $v\not\equiv 0$ in $\R^N$ and $v$ is antisymmetric with $v \ge 0$ in $H$, there exists a bounded set 
$M \subset H$ of positive measure with $x_o \not \in \overline M$ and such that 
\begin{equation}
  \label{eq:def-delta}
\delta:= \inf_{M} v >0.
\end{equation}
By Lemma~\ref{3-mengen-k}, we may fix $0<r< \frac{1}{4}\dist(x_0,[\R^N \setminus H] \cup M)$ such  that $\Lambda_1(B_{2r}(x_0)) > \|c\|_{L^{\infty}(U) }$. Next, we fix a function $f\in \cC^{2}_{c}(\R^{N})$ such that $0 \le f \le 1$ on $\R^N$ and 
\begin{equation*}
f(x):=\left\{\begin{aligned} 1, &&\text{ for $|x-x_0|\leq r$,}\\
0, &&\text{ for $|x-x_0|\geq 2r$.}\\
\end{aligned}\right.
\end{equation*}
Moreover we define 
$$
w: \R^N \to \R, \qquad w(x):=f(x)-f(\bar{x})+ a\bigl[1_M(x)-1_{M} (\bar x))\bigr],
$$
where $a>0$ will be fixed later. We also put $U_0:= B_{2r}(x_0)$ and $U_0':= B_{3r}(x_0) \cup Q(B_{3r}(x_0))$. Note that the function $w$ is antisymmetric and satisfies 
\begin{equation}
  \label{eq:ineq-w}
w\equiv 0 \quad \text{on $H \setminus (U_0 \cup M)$,} \quad \qquad w \equiv  a 
\quad \text{on $M$.}
\end{equation}
 We claim that $w \in \cH(U_0')$. Indeed, by Proposition~\ref{3-prel-dense}(i) 
we have $f - f \circ Q \in \cD(\R^N) \cap L^2(\R^N) \subset \cH(U_0')$, whereas 
$1_M - 1_{Q(M)} \in \cH(U_0')$ since $M$ is bounded and $[M \cup Q(M)] \cap U_0'= \emptyset$. \\
Next, let $\varphi\in \cD(U_0)$, $\varphi\geq0$. 
By Proposition~\ref{3-prel-dense}(ii) we have
\begin{equation}
\cJ(f,\varphi) \leq C \int_{U_0}\varphi(x)\ dx 
\end{equation}
with $C=C(f)>0$ independent of $\phi$.  Since 
$$
f(\bar x) \phi(x) = 1_{M}(x) \phi(x)= 1_{Q(M)}(x) \phi(x) = 0 \qquad \text{for every $x \in \R^N$,}
$$
we have 
\begin{align*}
 \cJ(w,\varphi)&=\cJ(f,\varphi) - \cJ(f \circ Q,\varphi)   +a\bigl[\cJ(1_M, \phi) -\cJ(1_{Q(M)}, \phi)
\bigr]\\
&\leq C \int_{U_0}\varphi(x)\ dx +\int_{U_0}\int_{Q(U_0)} \varphi(x)f(\overline y) J(x-y)\ dydx\\
 & -a\bigl[ \int_{U_0}\int_{M}\varphi(x)J(x-y)\ dydx -\int_{U_0}\int_{Q(M)}\varphi(x)J(x-y)\ dydx \bigr] \\
&\leq \Bigl( C + \sup_{x \in U_0} \int_{Q(U_0)} J(x-y)\,dy \Bigr) 
 \int_{U_0}\varphi(x)\ dx  -a\int_{U_0}\varphi(x) \int_{M}[J(x-y)-J(x-\bar{y})]\ dydx\\
&\le C_a \int_{U_0} \phi(x)\,dx
\end{align*}
with 
$$
C_a := C + \sup_{x \in U_0} \int_{Q(U_0)} J(x-y)\,dy - a \inf_{x \in U_0} \int_M (J(x-y)-J(x-\bar{y}))\ dy\in \R
$$
Since $\overline U_0 \subset H$, (\ref{sym-Q-J-2-strict}) and the continuity of the function $x \mapsto \int_M (J(x-y)-J(x-\bar{y}))\ dy$ on $\overline U_0$ imply that 
\[
\inf_{x \in U_0} \int_M (J(x-y)-J(x,\bar{y}))\ dy>0
\]
Consequently, we may fix $a>0$ sufficiently large such that $C_a \le -\|c\|_{L^\infty(U_0)}$. Since $0 \le w \le 1$ in $U_0$, we then have 
\begin{equation}
\label{eq:w-ineq}
\cJ(w,\varphi)  \le -\|c\|_{L^\infty(U_0)} \int_{U_0}\varphi(x)\,dx \leq \int_{U_0}c(x)w(x)\varphi(x)\ dx.
\end{equation}
We now consider the function $\tilde v:=v-\frac{\delta}{a} w \in \cH(U_0')$, which by (\ref{eq:def-delta}) and (\ref{eq:ineq-w}) satisfies $\tilde v \ge 0$ on $H \setminus U_0$.  Hence, by assumption and (\ref{eq:w-ineq}), $\tilde v$ is an antisymmetric supersolution of the problem  
\begin{equation}\label{linear-prob-special}
I\tilde v = c(x)\tilde v\quad \text{ in $U_0$,}\qquad  \tilde v \equiv 0 \quad \text{on $H \setminus U_0$}
\end{equation}
Since $\|c\|_{L^{\infty}(U_0)}<\Lambda_1(U_0)$, Proposition \ref{4-elliptic-max1} implies that $\tilde v \ge 
0$ a.e. in $U_0$, so that $v \geq \frac{\delta}{a} w = \frac{\delta}{a}  >0$ a.e. in $B_{r}(x_0)$. This ends the proof.
\end{proof}

\section{Proof of the main symmetry result }\label{mr}

In this section we complete the proof of Theorem \ref{sec:goal}.  So throughout this section, we assume that  $J:\R^N  \setminus \{0\} \to[0,\infty)$ is even and satisfies $(J1)$ and $(J2)$, $\Omega \subset \R^N$ satisfies $(D)$ and the nonlinearity $f$ satisfies $(F_1)$ and $(F_2)$. Moreover, we let $u \in L^\infty(\Omega) \cap \cD(\Omega)$ be a nonnegative solution of $(P)$.  For $\lambda \in \R$, we consider the open affine half space 
$$
H_{\lambda}:=\left \{
  \begin{aligned}
&\{x\in\R^{N}\;:\; x_{1}>\lambda\}\qquad \text{if $\lambda \ge  0$;}\\
 &\{x\in\R^{N}\;:\; x_{1}<\lambda\}\qquad \text{if $\lambda <  0$.}\\   
  \end{aligned}
\right.
$$
Moreover, we let $Q_{\lambda}:\R^{N}\to\R^{N}$ denote the reflection at $\partial H_{\lambda}$, i.e. $Q_{\lambda}(x)=(2\lambda-x_{1},x')$. By Remark~\ref{sec:equality-monotonicity}, we may assume without loss of generality that (\ref{eq:adj-measure}) holds. As noted in Remark~\ref{symmetry-need}, $J$ therefore satisfies the symmetry and monotonicity conditions (\ref{sym-Q-J-1}) and (\ref{sym-Q-J-2-strict}) with $H$ replaced by $H_\lambda$ for $\lambda \not=0$.  Let $\ell:=\sup \limits_{x\in \Omega}x_1$. 
Setting $\Omega_{\lambda}:= H_\lambda \cap \Omega$ for $\lambda \in \R$,  we note that $Q_{\lambda}(\Omega_\lambda)\subset \Omega$ for all $\lambda\in (-\ell,\ell)$ and $Q_{0}(\Omega)=\Omega$ as a consequence of assumption 
$(D)$. Then for all $\lambda\in (-\ell,\ell)$,  Remark~\ref{3-anti} implies that $v_{\lambda}:=u \circ Q_\lambda-u \in \cD(\R^N) \cap L^2(\R^N)$ is an antisymmetric supersolution of
the problem 
\begin{equation}\label{linear-prob-lambda}
Iv = c_\lambda (x)v \quad \text{ in $\Omega_\lambda$,}\qquad  v \equiv 0 \quad \text{on $H_\lambda  \setminus \Omega_\lambda$}
\end{equation}
with 
\[
c_\lambda \in L^{\infty}(\Omega_\lambda)\quad \text{given by}\quad  
c_\lambda(x)= \left \{
  \begin{aligned}
  &\frac{f(x,u(Q_\lambda(x)))-f(x,u(x))}{v_\lambda(x)},&&\qquad v_\lambda (x) \not= 0;\\
  &0, &&\qquad v_\lambda(x)= 0.
  \end{aligned}
\right.
\]
Note that, as a consequence of $(F1)$ and since $u \in L^\infty(\Omega)$,  we have 
$$
c_\infty:= \sup_{\lambda \in (-\ell,\ell)}\|c_\lambda\|_{L^\infty(\Omega_\lambda)} < \infty.
$$
We now consider the statement
\[
 (S_{\lambda})\qquad \underset{K}{\essinf} \:v_{\lambda}> 0 \qquad \text{for every compact subset $K \subset \Omega_{\lambda}$.}
\]
Assuming that $u \not\equiv 0$ from now on, we will show $(S_{\lambda})$ for all $\lambda\in (0,\ell)$. 
Since $|\Omega_{\lambda}| \to 0$ as $\lambda\to \ell$,  Lemma \ref{3-mengen-k}  implies that there exists $\epsilon \in (0,\ell)$  such that $\Lambda_1(\Omega_{\lambda})>c_\infty$ for all $\lambda\in [\epsilon,\ell)$. Applying Proposition \ref{4-elliptic-max1} we thus find that 
\begin{equation}
  \label{eq:almostS-lambda}
  v_\lambda \ge 0 \quad \text{a.e. in $H_\lambda\quad$ for all $\lambda \in [\epsilon,\ell)$.}
\end{equation}
We now show\\[0.1cm] 
{\em Claim 1: If $v_\lambda \ge 0$ a.e. in $H_\lambda$ for some $\lambda \in (0,\ell)$, then $(S_{\lambda})$ holds.}\\[0.1cm] 
To prove this, by Proposition~\ref{hopf-simple2} it suffices to show that $v_{\lambda}\not \equiv 0$ in $\R^N$. If, arguing by contradiction, $v_{\lambda}\equiv0$ in $\R^N$, then $\partial H_{\lambda}$ is a symmetry hyperplane of $u$. Since $\lambda \in (0,\ell)$ and $u \equiv 0$ in $\R^N \setminus \Omega$, we then have $u\equiv 0$ in the nonempty set $\Omega_{-\ell+2\lambda}$.  Setting $\lambda'= -\ell+\lambda$, we thus infer that $v_{\lambda'} \equiv 0$ in $\Omega_{\lambda'}$. Consequently, $v_{\lambda'}\equiv 0$ in $\R^{N}$ by Proposition \ref{hopf-simple2}.  Thus $u$ has the two different parallel symmetry hyperplanes $\partial H_{\lambda}$ and $\partial H_{\lambda'}$. Since $u$ vanishes outside a bounded set, this implies that $u\equiv 0$, which is a contradiction. Thus Claim 1 is proved.\\[0.1cm]
Next we show\\[0.1cm]
{\em Claim 2: If $(S_{\lambda})$ holds for some $\lambda \in (0,\ell)$, then there is $\delta\in(0,\lambda)$ such that $(S_{\mu})$ holds for all $\mu \in (\lambda-\delta,\lambda)$.}\\[0.1cm]
To prove this, suppose that $(S_{\lambda})$ holds for some $\lambda \in (0,\ell)$. Using Lemma~\ref{3-mengen-k}, we fix $s \in (0, |\Omega_\lambda|)$ such that $\Lambda_1(s)> c_\infty$, which implies that $\Lambda_1(U)>c_\infty$ for all open sets $U\subset\R^N$ with $|U|\leq s$.  
Since $\Omega$ is bounded,  we may also fix $\delta_0 >0 $ such that 
$$
|\Omega_{\mu} \setminus \Omega_{\mu+\delta_0}|<s/2\qquad \text{for all $\mu \ge 0$.}
$$
By Lusin's Theorem, there exists a compact subset $K \subset \Omega$ such that $|\Omega \setminus K|< s/4$ and such that the restriction 
$u|_{K}$ is continuous. For $\mu \ge 0$, we now consider the compact set
$$
K_\mu:= \overline \Omega_{\mu+\delta_0} \cap K \cap Q_\mu(K) \:\subset \:K \cap \Omega_\mu $$
and the open set $U_\mu:=\Omega_\mu \setminus K_\mu$. Note that  
\begin{equation}
  \label{eq:additional-1}
 |U_\mu| \le |\Omega_\mu \setminus \Omega_{\mu+\delta_0}|+ |\Omega_\mu \setminus K| + 
|\Omega_\mu \setminus Q_\mu(K)| \le \frac{s}{2} + 2 |\Omega \setminus K| < s \qquad \text{for $\mu \ge 0$. }
\end{equation}
As a consequence, for $0 \le \mu \le \lambda$ we have $|K_\mu| > |\Omega_\mu|-s \ge |\Omega_\lambda|-s>0$  and thus $K_\mu \not = \varnothing$. Property $(S_\lambda)$ and the continuity of $u|_K$ imply that  $\min \limits_{K_\lambda} v_\lambda>0$. Thus, again by the continuity of $u|_K$, there exists $\delta \in (0,\min \{\lambda,\delta_0\})$ such that 
 $$
\min_{K_\mu} v_\mu >0 \qquad \text{for all $\mu \in[\lambda-\delta,\lambda]$.}
$$
Consequently, for $\mu \in (\lambda-\delta,\lambda)$, the function $v_{\mu}$ is an antisymmetric supersolution of the problem 
$$
Iv=c_{\mu}(x)v \quad \text{in $U_\mu$,} \qquad v \equiv 0 \quad \text{on $H_{\mu}  \setminus U_\mu$, }
$$
whereas $\Lambda_1(U_\mu)>c_\infty$ by (\ref{eq:additional-1}) and the choice of $s$. Hence $v_{\mu}\geq 0$ in $H_{\mu}$ by Proposition \ref{4-elliptic-max1}, and thus $(S_{\mu})$ holds by Claim 1. This proves Claim 2.\\[0.1cm]
To finish the proof, we consider 
$$
\lambda_{0}:=\inf\{\tilde{\lambda} \in (0,\ell) \;:\; \text{ $(S_{\lambda})$ holds for all $\lambda \in (\tilde{\lambda},\ell) $}\} \quad \in \;[0,\ell).
$$ 
We then have $v_{\lambda_0} \ge 0$ in $H_{\lambda_0}$. Hence Claim 1 and Claim 2 imply that $\lambda_{0}=0$. Since the procedure can be repeated in the same way starting from $-\ell$,  we find that $v_{0} \equiv 0$.  Hence the function $u$ has the asserted symmetry and monotonicity properties.\\
It remains to show (\ref{eq:positivity-thm-1-1}). 
So let $K \subset \Omega$ be compact. Replacing $K$ by $K \cup Q_0(K)$ if necessary, we may assume that $K$ is symmetric with respect to $Q_0$. Let $K':= \{x \in K\::\: x_1 \le 0\}$. Since for $\lambda>0$ sufficiently small $Q_\lambda(K')$ is a compact subset of $\Omega_\lambda$, the property $(S_\lambda)$ and the symmetry of $u$ then imply that 
$$
\underset{K}{\essinf}\, u =  \underset{K'}{\essinf}\, u  \ge \underset{Q_\lambda(K')}\essinf\, v_\lambda >0,
$$ 
as claimed in (\ref{eq:positivity-thm-1-1}).

\section{Proof of a variant symmetry result}
\label{sec:vari-symm-result}

In this section we prove Theorem~\ref{sec:vari-symm-result-1}, which is concerned with the class of even kernel functions  satisfying $(J2)'$ in place of $(J2)$.  Throughout this section, we consider a symmetric kernel $J: \R^N \setminus \{0\} \to [0,\infty)$ satisfying $(J1)$.  We fix an open affine half space $H \subset \R^N$, and we consider the notation of Section~\ref{mp}. Moreover, we assume the symmetry and monotonicity assumptions (\ref{sym-Q-J-1}) and (\ref{sym-Q-J-2}), so that Lemma~\ref{sec:linear-problem-tech} and Proposition~\ref{4-elliptic-max1} are available. In order to derive a variant of the strong maximum principle given in Proposition~\ref{hopf-simple2}, we introduce the following strict monotonicity condition: 
\begin{equation}
\text{There exists $r_0>0$ such that $J(x-y) > J(x- \bar y)$ for all $x,y \in H$ with $|x-y| \le r_0$}  \label{sym-Q-J-2-strict-loc}
 \end{equation}
We then have the following.

\begin{prop}\label{hopf-simple2-variant}
Assume that $J$ satisfies $(J1)$, (\ref{sym-Q-J-1}), (\ref{sym-Q-J-2}) and (\ref{sym-Q-J-2-strict-loc}). Moreover, let $U \subset H$ be  a subdomain and $c \in L^\infty(U)$. 
Furthermore, let  $v$  be an antisymmetric supersolution of (\ref{linear-prob})  
such that $v\geq0$ a.e. in $H$.\\
Then either $v \equiv 0$ a.e. in a neighborhood of $\overline U$, or 
$$
\underset{K}\essinf\: v >0 \qquad \text{for every compact subset $K  \subset U$.}
$$
\end{prop}
We stress that, in contrast to Proposition~\ref{hopf-simple2}, we require connectedness of $U$ here. 

\begin{proof}
Let $W$ denote the set of points $y \in U$ such that $\underset{B_{r}(y)}\essinf\: v >0$ for $r>0$ sufficiently small, and let $r_0>0$ be as in (\ref{sym-Q-J-2-strict-loc}). We claim the following.\\
\begin{equation}
  \label{claim-1}
\text{If $x_0 \in U$ is such that $v \not \equiv 0$ in $B_{\frac{r_0}{2}}(x_0)$, then $x_0 \in W$.}
\end{equation}
To prove this, let $x_0 \in U$ be such that $v \not \equiv 0$ in $B_{\frac{r_0}{2}}(x_0)$. Then there exists a bounded set 
$M \subset H \cap B_{\frac{r_0}{2}}(x_0)$ of positive measure with $x_0 \not \in \overline M$ and such that 
\begin{equation}
\delta:= \inf_{M} v >0 
\end{equation}
By Lemma~\ref{3-mengen-k}, we may fix $0<r< \frac{1}{4}\min \{r_0\,,\, \dist(x_0,[\R^N \setminus H] \cup M)\}$ such  that $\Lambda_1(B_{2r}(x_0)) > \|c\|_{L^{\infty}(U) }$. Next, we put $U_0:= B_{2r}(x_0)$ and $U_0':= B_{3r}(x_0) \cup Q(B_{3r}(x_0))$. Moreover, we define the functions $f\in C^{2}_{c}(\R^{N})$ and $w \in \cH(U_0')$, depending on $a>0$, as in the proof of Proposition~\ref{hopf-simple2}. 
As noted there,  $w$ is antisymmetric and satisfies 
\begin{equation}
  \label{eq:ineq-w2}
w\equiv 0 \quad \text{on $H \setminus (U_0 \cup M)$,} \quad \qquad w \equiv  a 
\quad \text{on $M$.}
\end{equation}
As in the proof of Proposition~\ref{hopf-simple2}, we also see that 
$$
 \cJ(w,\varphi) \le C_a \int_{U_0} \phi(x)\,dx  
\qquad \text{for all $\varphi \in \cD(U_0), \phi \ge 0$}
$$
with 
$$
C_a := C + \sup_{x \in U_0} \int_{Q(U_0)} J(x-y)\,dy - a \inf_{x \in \overline U_0} \int_M (J(x-y)-J(x-\bar{y}))\ dy
$$
Since $\overline U_0 \subset H \cap B_{\frac{r_0}{2}}(x_0) $ and $M \subset H \cap B_{\frac{r_0}{2}}(x_0)$, (\ref{sym-Q-J-2-strict-loc}) and the continuity of the function $x \mapsto \int_M (J(x-y)-J(x-\bar{y}))\ dy$ on $\overline U_0$ imply that 
\[
\inf_{x \in \overline U_0} \int_M (J(x-y)-J(x,\bar{y}))\ dy>0
\]
Hence we may proceed precisely as in the proof of Proposition~\ref{hopf-simple2} to prove that $v \geq  \frac{\delta}{a}  >0$ a.e. in $B_{r}(x_0)$ for $a>0$ sufficiently large, so that $x_0 \in W$. Hence (\ref{claim-1}) is true.\\
From (\ref{claim-1}) it immediately follows that $W$ is both open and closed in $U$. Moreover, if $v \not \equiv 0$ in $\{x \in H\::\: \dist(x,U) < \frac{r_0}{2}\}$, then $W$ is nonempty and therefore $W=U$ by the connectedness of $U$. This ends the proof.
\end{proof}

Next we complete the proof of Theorem~\ref{sec:vari-symm-result-1}.   So throughout the remainder of this section, we assume that  $J:\R^N  \setminus \{0\} \to[0,\infty)$ is even and satisfies $(J1)$ and $(J2)'$, $\Omega \subset \R^N$ satisfies $(D)$ and the nonlinearity $f$ satisfies $(F_1)$ and $(F_2)$.  Moreover, we let $u \in L^\infty(\Omega) \cap \cD(\Omega)$ denote an a.e. positive solution of $(P)$. For $\lambda \in \R$, we let $H_\lambda$, $Q_\lambda$, $\Omega_\lambda$, $c_\lambda$ and $v_\lambda$ 
be defined as in Section~\ref{mr}, and again we put $\ell:=\sup \limits_{x\in \Omega}x_1$. As a consequence of $(J1)$ and $(J2)'$, we may assume that  $J$ satisfies~(\ref{sym-Q-J-1})~(\ref{sym-Q-J-2}) and (\ref{sym-Q-J-2-strict-loc}) with $H$ replaced by $H_\lambda$ for $\lambda \not=0$ (the argument of Remark~\ref{symmetry-need} still applies). As in Section~\ref{mr}, we then consider the statement
\[
 (S_{\lambda})\qquad \underset{K}{\essinf} \:v_{\lambda}> 0 \qquad \text{for every compact subset $K \subset \Omega_{\lambda}$.}
\]
We wish to show $(S_{\lambda})$ for all $\lambda\in (0,\ell)$. As in Section~\ref{mr}, we find $\epsilon \in (0,\ell)$  such that 
\begin{equation}
  \label{eq:almostS-lambda-variant}
  v_\lambda \ge 0 \quad \text{a.e. in $H_\lambda\quad$ for all $\lambda \in [\epsilon,\ell)$.}
\end{equation}
We now show\\[0.1cm] 
{\em Claim 1: If $v_{\lambda} \ge 0$ a.e. in $H_{\lambda}$ for some $\lambda \in (0,\ell)$, then $(S_{\lambda})$ holds.}\\[0.1cm] 
To prove this, we argue by contradiction. If $(S_\lambda)$ does not hold, then, by Proposition~\ref{hopf-simple2-variant}, there exists a connected component $\Omega'$ of $\Omega_\lambda$ and a  neighborhood $N$ of $\overline{\Omega'}$ such that $v_{\lambda}\equiv 0$ in $N$. However, since $\lambda \in (0,\ell)$, the set $\tilde N:= Q_\lambda(N \setminus \Omega) \cap \Omega$ has positive measure and  $v_\lambda \equiv 0$ in $\tilde N$ by the antisymmetry of $v_\lambda$. However, $v \equiv -u$ on $\tilde N$, so $u \equiv 0$ a.e. on $\tilde N$, contrary to the assumption that $u >0$ a.e. in $\Omega$. Thus Claim 1 is proved.\\[0.1cm]
Precisely as in Section~\ref{mr} we may now show\\[0.1cm] 
{\em Claim 2: If $(S_{\lambda})$ holds for some $\lambda \in (0,\ell)$, then there is $\delta\in(0,\lambda)$ such that $(S_{\mu})$ holds for all $\mu \in (\lambda-\delta,\lambda)$.}\\[0.1cm]
Moreover, based on (\ref{eq:almostS-lambda-variant}), Claim 1 and Claim 2, we may now finish the proof of Theorem~\ref{sec:vari-symm-result-1} precisely as in the end of Section~\ref{mr}.

\bibliographystyle{amsplain}

\end{document}